\documentclass{article}
\usepackage{amsmath,color,amsfonts,amsthm}
\usepackage{amssymb,enumerate,verbatim,enumitem}
\usepackage{hyperref}
\usepackage[bottom,hang,flushmargin]{footmisc}
\usepackage{tikz}
\usetikzlibrary{calc}

\usepackage{a4wide}

\makeatletter
\newcommand{\globalcolor}[1]{%
  \color{#1}\global\let\default@color\current@color
}
\makeatother

\newtheorem{lemma}{Lemma}[section]
\newtheorem{corollary}[lemma]{Corollary}

\newtheorem{theorem}[lemma]{Theorem}

\theoremstyle{definition}

\global\long\def\E{\mathbb{E}}

\global\long\def\P{\mathbb{P}}

\date{}

\title{\vspace{-0.7cm} $C_4$-free subgraphs with large average degree}
\author{R. Montgomery\thanks{School of Mathematics,
University of Birmingham,
Birmingham,
B15 2TT,
UK. r.h.montgomery@bham.ac.uk.}
, A. Pokrovskiy\thanks{Department of Economics, Mathematics, and Statistics, Birkbeck College, University of London. dr.alexey.pokrovskiy@gmail.com.}
, and B. Sudakov\thanks{Department of Mathematics, ETH, 8092 Zurich, Switzerland. benjamin.sudakov@math.ethz.ch. Research supported in part by SNSF grant 200021-175573.}
}

\begin{document}
\maketitle

\begin{abstract}
Motivated by a longstanding conjecture of Thomassen, we study how large the average degree of a graph needs to be to imply that it contains a $C_4$-free subgraph with average degree at least $t$. K\"uhn and Osthus showed that an average degree bound which is double exponential in t is sufficient. We give a short proof of this bound, before reducing it to a single exponential. That is, we show that any graph $G$ with average degree at least $2^{ct^2\log t}$ (for some constant $c>0$) contains a $C_4$-free subgraph with average degree at least $t$. Finally, we give a construction which improves the lower bound for this problem, showing that this initial average degree must be at least $t^{3-o(1)}$.
\end{abstract}

\section{Introduction}
The girth of a graph $G$, denoted $g(G)$, is the length of a shortest cycle in $G$.  A celebrated conjecture of Thomassen~\cite{Thomassen83} from 1983 says that, for each $t$ and $g$, there is some $f(t,g)$ such that every graph $G$ with average degree $d(G) \geq f(t,g)$ contains a subgraph with girth at least $g$ and average degree at least $t$. 

This is straightforward for regular graphs, or, more generally, for graphs whose maximum degree $\Delta(G)$ is at most a constant multiple of $d(G)$. Indeed, given an $n$-vertex $d$-regular graph $G$, take a subgraph $H$ of $G$ by deleting every edge with probability $p=d^{-(g-1)/g}/2$. The expected number of edges in $H$ is $ndp/2=n d^{1/g}/4$ and the expected number of cycles with length at most $g$ in $H$ is at most $nd^{g-1}p^g=n/2^{g-1}$. Thus, deleting a vertex from each short cycle of a typical such $H$ gives a graph with average degree at least $d^{1/g}/4$ and girth larger than $g$.

This argument can, of course, be used for any graph containing an almost regular subgraph with high average degree. Pyber, R\"odl and Szemer\'edi \cite{PRS} showed that any graph $G$ whose average degree is at least logarithmic in $\Delta(G)$ has an $r$-regular subgraph (with $r$ growing together with $d(G)$). This can therefore be used to prove Thomassen's conjecture for such graphs. On the other hand, Pyber, R\"odl and Szemer\'edi \cite{PRS} also proved that there are graphs $G$ with average degree at least $c\log \log \Delta(G)$ that do not contain even a $3$-regular subgraph. This shows that one can not prove Thomassen's conjecture through reduction to the regular case. However, more progress has been made for graphs where the maximum degree is bounded by a function of the average degree. Indeed, Dellamonica and R\"odl \cite{DR} proved the conjecture for each graph $G$ with average degree at least $\alpha (\log \log \Delta(G))^\beta$, for some constants $\alpha$ and $\beta$. Nevertheless, in general Thomassen's conjecture remains widely open.

For general graphs, as it is well known that every graph contains a bipartite (and hence odd cycle free) subgraph with at least half of the edges, Thomassen's conjecture is trivial for each $g\leq 4$. The only non-trivial case of Thomassen's conjecture obtained so far was by the breakthrough paper of K\"uhn and Osthus~\cite{KK04}, who proved the case where $g=6$. More precisely, they showed that, for some constant $c>0$, every graph $G$ with average degree at least $2^{2^{ct^3}}$ contains a subgraph with average degree at least $t$ and girth at least 6. 
An alternative proof of this, with a similar double exponential bound, was later given by Dellamonica, Koubek, Martin, and R\"odl~\cite{DKMV11}. Using their approach, McCarty \cite{McCarty} recently proved that a bipartite graph $G$ with large average degree and no copy of $K_{t,t}$ contains an \emph{induced} $C_4$-free subgraph with many edges. The assumption that $G$ is bipartite can be further removed, using the result from \cite{KLST}, which says that an $H$-free graph with average degree $d$ contains an induced bipartite subgraph with average degree at least $\log^{1-o(1)} d$.

With the establishment of an upper bound for the $g=6$ case of Thomassen's conjecture, it is natural to ask whether this double exponential bound can be improved. In this paper, we study this, starting with a very short proof of the result of K\"uhn and Osthus. Building on this, we then give the following single exponential bound.

\begin{theorem}\label{shortC4free} There exists $c>0$ such that, for each $t\in \mathbb{N}$, every graph with average degree at least $2^{ct^2\log t}$ contains a graph with average degree at least $t$ and girth at least 6.
\end{theorem}

On the other hand, it is known that there are graphs with average degree $(1+o(1))t^2$ which contain no subgraphs with average degree at least $t$ and girth bigger than 4. Indeed, this follows easily from well-known bounds on the Tur\'an numbers of the $4$-cycle proved by 
Erd\H{o}s, R\'enyi and S\'os \cite{ERS} (see also 10.36 (a) in \cite{lovaszCPE}). They showed that the maximum number of edges in any $n$-vertex $C_4$-free graph is at most $(1/2+o(1))n^{3/2}$. Therefore the complete graph on $n=(1+o(1))t^2$ vertices has average degree $(1+o(1))t^2$ yet no $C_4$-free subgraph with average degree at least $t$. Thus, the bound in Theorem~\ref{shortC4free} must be at least quadratic in $t$. However, the discussion above on techniques for almost regular graphs suggests we should look to graphs with irregular degree sequences to improve this. By giving a new construction, we show the bound in Theorem~\ref{shortC4free} must be at least roughly cubic in $t$, as follows.

\begin{theorem}\label{exampleC4free} There is a constant $c$ such that, for all $t$ there exists a graph with average degree at least $t^3$ yet no $C_4$-free subgraph with average degree at least $ct\log t$.
\end{theorem}

\vspace{0.25cm}
\noindent{\bf Notation.}\, Given a graph $G$ we denote by $\Delta(G)$ the maximum degree of $G$, by $e(G)$ the number of edges of $G$ and by $d(G)$ the average degree of $G$. For a vertex $x$ we denote by $N_G(x)$ the set of neighbours of $x$ and by $N^2_G(x)$ the set of vertices in $G$ of distance two from $x$.
For any pair of vertices $x, y \in V(G)$, the codegree $d_G(x,y)$ is the number of common neighbours of $x$ and $y$ in $G$. Given $A\subset V(G)$ and $x\in V(G)$ we denote by $d(x, A)$ the number of neighbours of $x$ in $A$. Finally, for a pair of disjoint vertex subsets $A, B$ in $G$ we denote by $e_G(A,B)$ the number of edges of $G$ between $A$ and $B$ and by  $G[A,B]$ the induced bipartite subgraph of $G$ with all such edges.

\section{Double exponential upper bound}
In this section, we give a short proof of the following slightly stronger form of the result of K\"uhn and Osthus~\cite{KK04}.

\begin{theorem}\label{veryshortC4free} There exists $c>0$ such that every graph with average degree at least $2^{2^{ct^2}}$ contains a graph with average degree at least $t$ and girth at least 6.
\end{theorem}

Since, as is well known, every graph contains a bipartite subgraph with at least half of the edges, we can assume that $G$ is bipartite. That is, we can assume the initial graph $G$ has girth at least four and no odd cycles. In order to obtain a subgraph of $G$ with girth at least 6, we need only to find a subgraph which has no $4$-cycles, that is, which is $C_4$-free.
The main idea of the proof (inspired by that of K\"uhn and Osthus) is to find either find a dense subgraph of $G$ with small codegrees, and hence few $4$-cycles, or a large complete bipartite subgraph of $G$. In the first case a further random subgraph is likely to be still dense but have no $4$-cycles, and in the second case we use the following well known 
construction of Reiman for Zarankiewicz's  problem (see 10.15a in \cite{lovaszCPE}). It is based on a projective plane and gives a dense $C_4$-free subgraph of the complete bipartite graph.

\begin{lemma}\label{bipartitecomp} If $s=4k^2$, then $K_{s,s}$ has a $C_4$-free subgraph with average degree at least $k$.
\end{lemma}

The next lemma, which is a main step in the proof, finds either a dense subgraph with small codegrees or can be used to build  a large complete bipartite subgraph, one vertex at a time.

\begin{lemma}\label{twocases} Let $G$ be a bipartite graph with vertex classes $A$ and $B$ and let $\lambda\geq 1$. Then, either
  \begin{enumerate}[label = (\arabic{enumi})]
    \item there is some vertex $v\in B$ and sets $A'\subset N(v)$ and $B'\subset B\setminus \{v\}$ such that $d(G[A',B'])\geq \lambda $, or
    \item there is a spanning subgraph $H\subset G$ with $d(H)\geq d(G)/(\lambda +1)$ and $d_H(x,y)\leq \lambda$, for all $x,y\in B$.
  \end{enumerate}
\end{lemma}
%%\begin{comment}
\begin{proof}
Let $G=G_0$, $n=|B|$ and label $B=\{v_1,\ldots,v_n\}$. For each $i=1,\ldots,n$, repeat the following. Given a spanning  subgraph $G_{i-1}$, let $V_i$ be the set of all vertices in $\{v_{i+1},\ldots,v_n\}$ which have at least $\lambda$ neighbours in the set $N_{G_{i-1}}(v_i)$. By definition, there are at least
$\lambda|V_i|$ edges of $G_{i-1}$ from $V_i$ to $N_{G_{i-1}}(v_i)$. If in addition $e_{G_{i-1}}(V_i,N_{G_{i-1}}(v_i))\geq \lambda d_{G_{i-1}}(v_i)$, then the induced subgraph $G_{i-1}[V_i,N_{G_{i-1}}(v_i)]$ has average degree at least $\lambda $. Therefore, (1) holds with $v=v_i$, $A'=N_{G_{i-1}}(v_i)$ and $B'=V_i$. Otherwise, if $e_{G_{i-1}}(V_i,N_{G_{i-1}}(v_i))< \lambda d_{G_{i-1}}(v_i)$, remove all the edges of $G_{i-1}$ between $V_i$ and $N_{G_{i-1}}(v_i)$ to form $G_i$. Note that $v_i\notin V_i$, so that no edges adjacent to $v_i$ are removed in this operation.

Suppose that (1) never holds, so that the process terminates with $H:=G_n\subset G_{n-1}\subset\ldots\subset G_0=G$. Now, for every $i<j$, by construction of $G_i$ we have $d_{G_i}(v_i,v_j)<\lambda$. Thus $d_H(v_i,v_j)<\lambda $ for each $i\neq j$. Most importantly, edges incident to $v_i$ in $G_{i-1}$ are never subsequently removed. Indeed, in step $i$ no edges incident to $v_i$ are removed (as $v_i\notin V_i$), and, after step $i$, $v_i$ has codegrees less than $\lambda$ with all $v_j, j>i$ and we only remove edges incident to pairs with high codegree. Therefore, $d_{G_{i}}(v_i)=d_{G_{i-1}}(v_i)$.

Since we removed fewer than $\lambda d_{G_{i-1}}(v_i)$ edges from $G_{i-1}$ to get $G_i$ we have $d_{H}(v_{i})= d_{G_{i-1}}(v_{i})> (e(G_{i-1})-e(G_i))/\lambda$.
Therefore,
\[
e(H)=\sum_{i\in [n]}d_{H}(v_{i})> \sum_{i\in [n]}\frac{e(G_{i-1})-e({G_i})}{\lambda}=\frac{e(G)-e(H)}{\lambda }.
\]
Hence, $e(H)> e(G)/(\lambda +1)$, and thus $d(H)> d(G)/(\lambda +1)$. I.e., $H$ satisfies (2).
\end{proof}

If a graph has small codegrees yet many edges, a typical random subgraph will be $C_4$-free with large average degree, as follows.

\begin{lemma}\label{newprobmethod} There is some $d_0$ such that the following holds for each $d\geq d_0$.
Let $G$ be a bipartite graph with vertex classes $A$ and $B$, with $|A|\geq |B|$, $d_G(x,y)\leq d^{1/5}$ for each $x,y\in B$ with $x\neq y$, $d(x)\leq d$ for each $x\in A$ and $d(G)\geq d^{3/4}$. Then, $G$ contains a $C_4$-free subgraph with average degree at least $d^{1/4}$.
\end{lemma}
\begin{proof}
Let $n=|A|$ and note that $n\leq |G|\leq 2n$. For each $v\in A$, there are at most $d^2$ distinct pairs of vertices $x,y\in N(v)$, and hence at most $d^{1/5}\cdot d^2=d^{11/5}$ copies of $C_4$ in $G$ containing $v$. Thus, $G$ contains at most $nd^{11/5}$ copies of $C_4$.
Let $H$ be a random subgraph of $G$ formed by including each edge independently at random with probability $p=3d^{-1/2}$.
Let $X$ be the number of copies of $C_4$ in $H$. As $e(G)\geq n\cdot d(G)/2\geq nd^{3/4}/2$, we have
\[
\E\big(e(H)- X\big)\geq p\cdot e(G)-nd^{11/5}p^4 \geq 3nd^{1/4}/2-81 nd^{1/5}\geq nd^{1/4}.
\]
Thus, there is some subgraph $H\subset G$ with $e(H)-X\geq nd^{1/4}$. As $|G|\leq 2n$, removing an edge from each $C_4$ in $H$ thus gives a $C_4$-free subgraph with average degree at least $d^{1/4}$.
\end{proof}
We apply this lemma through the following corollary.
\begin{corollary}\label{cortwocasesnew} There is some $d_0$ such that the following holds for each $d\geq d_0$.
  Let $G$ be a bipartite graph with average degree $d\geq t^4$ which contains no $C_4$-free graph with average degree at least $t$. Then, there is some vertex $v\in V(G)$ and sets $A'\subset N(v)$ and $B'\subset V(G)\setminus \{v\}$ such that $d(G[A',B'])\geq d^{1/5}$.
\end{corollary}
\begin{proof} 
Remove from $G$ one by one vertices of degree less than $d/2$. This does not decrease its average degree, and produces a subgraph $G_0\subset G$ with $d(G_0)\geq d(G)$ and $\delta(G_0)\geq d/2$. Suppose this (bipartite) subgraph has vertex classes $A$ and $B$ with $|A|\geq |B|$. For each $v\in A$, select $\lceil d/2\rceil$ incident edges and add them to $G_1$. Then, $e(G_1)\geq d|A|/2$ and  $d_{G_1}(v) \leq d$, for each $v\in A$.
By Lemma~\ref{twocases} with $\lambda=d^{1/5}$, if $G_1$ does not satisfy the assertion of the corollary, then $G_1$ contains a spanning subgraph $H\subset G_1$ with $d(H)\geq d(G_1)/(d^{1/5}+1)\geq d^{3/4}$ and $d_H(x,y)\leq d^{1/5}$, for all $x,y \in B$. Then, by Lemma~\ref{newprobmethod}, $H$ contains a $C_4$-free subgraph with average degree at least $d^{1/4}\geq t$, a contradiction.
\end{proof}
Applying this corollary iteratively, we can now prove Theorem~\ref{veryshortC4free}.

\begin{proof}[Proof of Theorem~\ref{veryshortC4free}] Suppose $t$ is large enough that Corollary~\ref{cortwocasesnew} holds for each $d\geq t^{5^{t^2}}$.
Suppose $G$ is bipartite, with average degree at least $t^{5^{9t^2}}$, and let $G_0=G$. Suppose, for contradiction, that $G_0$ contains no $C_4$-free subgraph with average degree at least $t$, and note that $d(G_0)\geq t^4$.

For each $i=1,\ldots,8t^2$, by Corollary~\ref{cortwocasesnew}, we can find a vertex $v_i$ and disjoint sets $A_i\subset N(v_i)$ and $B_i\subset V(G_{i-1})\setminus \{v_i\}$ and a   graph $G_i=G_{i-1}[A_i,B_i]$ with average degree at least $t^{5^{9t^2-i}}$. Moreover, note that, for each $j<i$, $A_i\subset N(v_j)$ or $B_i\subset N(v_j)$.

Let $s=4t^2$. Thus, we have vertices $v_1,\ldots,v_{2s}$ and a graph $G_{2s}$ with average degree at least $t^{5^{t^2}}>s$ and vertex sets $A_{2s}$ and $B_{2s}$ such that, for each $i\in [s]$, either $A_{2s}\subset N(v_i)$ or $B_{2s}\subset N(v_i)$.
Relabelling, we can assume that we have vertices $v_1,\ldots,v_s$ with $A_{2s}\subset N(v_i)$ for each $i\in [s]$. As $d(G_{2s})\geq s$, $|A_{2s}|\geq s$, and therefore $G_{2s}$, and hence $G$, contains a copy of $K_{s,s}$. Then, by Lemma~\ref{bipartitecomp}, $G$ contains a $C_4$-free subgraph with average degree at least $t$, a contradiction.
\end{proof}

\section{Proof of the main result}
In this section we prove Theorem \ref{shortC4free}. We still use an iterative procedure which finds either a dense $C_4$-free subgraph or makes progress towards a complete bipartite graph (cf.\ Corollary~\ref{cortwocasesnew}). However, now the average degree will decrease much less on each iteration. Instead of passing from average degree $d$ to average degree at least $d^{1/5}$ (which led to our double exponential bound), the average degree $d$ decreases to only average degree at least $d/50t$ at each iteration, where $t$ is the average degree we are aiming for. To do this we wish to apply Lemma~\ref{twocases} to a graph $G$ with $\lambda=d(G)/50t$ instead of $\lambda=d(G)^{1/5}$. If (1) in Lemma~\ref{twocases} holds, we can iterate as before. If (2) holds, we need to do more work. In general, the conditions in (2) are not strong enough for our techniques to find the required dense $C_4$-free subgraph, but we can do this if, in addition, the subgraph is a very unbalanced bipartite graph (see Lemma~\ref{newdiff}). Therefore, as this subgraph is always spanning, we ensure this by only applying Lemma~\ref{twocases} to very unbalanced bipartite graphs. Fortunately, in each dense bipartite graph we can find either a very unbalanced dense bipartite subgraph, or a dense bipartite subgraph whose maximum degree is at most polynomial in the average degree (see Lemma~\ref{splitbigdeg}). This latter case can be solved using the techniques for nearly regular graphs discussed in the introduction (see also Lemma~\ref{randsublocal}). Therefore, by applying Lemma~\ref{splitbigdeg} before applying Lemma~\ref{twocases} in each iteration, we gain the additional property that the graph is very unbalanced, which we use when (2) in Lemma~\ref{twocases} holds.

We start with Lemma~\ref{splitbigdeg}, which finds a dense subgraph that is either very unbalanced  or has low maximum degree.

\begin{lemma}\label{splitbigdeg}
Every bipartite graph $G$ with $d(G)\geq k\geq 2$ contains a subgraph $H$ with vertex classes $A$ and $B$ such that $d(H)\geq k/4$, $d_H(v)\leq k$ for each $v\in A$, and either
\begin{enumerate}[label = (\arabic{enumi})]
\item $|A|\geq k^6|B|$, or
\item $\Delta(H)\leq k^7$.
\end{enumerate}
\end{lemma}
%\begin{comment}
\begin{proof}
Remove from $G$ one by one vertices of degree less than $k/2$. This does not decrease its average degree, and produces a subgraph $G_0\subset G$ with $d(G_0)\geq d(G)$ and $\delta(G_0)\geq k/2$. Suppose this (bipartite) subgraph has vertex classes $A$ and $B_0$ with $|A|\geq |B_0|$. For each $v\in A$, select $\lceil k/2\rceil$ incident edges and add them to $G_1$. Then, $e(G_1)\geq k|A|/2$ and  $d_{G_1}(v) \leq k$, for each $v\in A$.

Let $B_1\subset B_0$ be the set of vertices with degree at least $k^7$ in $G_1$, and let $B_2=B\setminus B_0$. Note that
\begin{eqnarray*}
d(G_1[B_1,A])+d(G_1[B_2,A])&=&\frac{2e(G_1[B_1,A])}{|B_1|+|A|}+\frac{2e(G_1[B_2,A])}{|B_2|+|A|}\\
&\geq& \frac{2e(G_1[B_1,A])+ 2e(G_1[B_2,A])}{|A|+|B_0|}
\geq \frac{2e(G_1)}{2|A|}\geq\frac{k}{2}.
\end{eqnarray*}
Therefore, either $d(G_1[B_1,A])\geq k/4$ or $d(G_1[B_2,A])\geq k/4$. If $d(G_1[B_2,A])\geq k/4$, then, letting $H=G_1[B_2,A]$, we have that $\Delta(H)\leq k^7$, and so (2) is satisfied.
On the other hand, if $d(G_1[B_1,A])\geq k/4$, then, letting $H=G_1[B_1,A]$, we have that $|B_1|k^7\leq \lceil k/2\rceil\cdot |A|$, so that (1) is satisfied.
\end{proof}

For graphs satisfying (2) in Lemma~\ref{splitbigdeg}, we show that with a small reduction in average degree we can reduce the maximum degree bound even further.

\begin{lemma}\label{maxdegreduced} 
Let $G$ be a bipartite graph with vertex classes $A$ and $B$ satisfying $\Delta(G)\leq k^7$, $d(G)\geq k/4$, and $d_G(v)\leq k$ for each $v\in A$. If $k$ is sufficiently large, then $G$ contains a subgraph with maximum degree at most $k$ and average degree at least $k/400\log k$.
\end{lemma}
%\begin{comment}
\begin{proof}
Note that we can assume that $G$ has no isolated vertices. Let $r=10\log k$, and, for each $i\in [r]$, let $B_i=\{v\in B:2^{i-1}\leq d_G(v)< 2^i\}$, noting that these sets partition $B$ as $2^r\geq k^7$. For each $i\in [r]$, let $G_i=G[A,B_i]$. As $e(G)\geq k|A|/8$, there must be some $i$ with $e(G_i)\geq k|A|/80\log k$.

Let $d=2^{i-1}$, so that, for each $v\in B$, $d\leq d_{G_i}(v) \leq 2d$. If $2d\leq k$, then $G_i$ satisfies the lemma, so assume that $2d>k$. Thus, $p:=k/4d<1/2$. Now, let $A'\subset A$ be chosen by including each vertex in $A$ independently at random with probability $p$. Let $B'$ be the set of vertices in $B_i$ with at most $k$ neighbours in $A'$ and let $H=G_i[B'\cup A']$. Then, by definition, $\Delta(H) \leq k$.

For each $v\in B_i$, its degree in $A'$ is binomially distributed with expectation $\E(d(v,A'))$ satisfying
$k/4=pd \leq \E(d(v,A')) \leq 2pd=k/2$. Therefore, using Chernoff's bound (see, e.g, Appendix A of \cite{AS}), we have
\[
\P(d(v,A')\geq k)\leq \P(d(v,A')\geq 2\cdot \E(d(v,A')))\leq 2\exp\left(-\E(d(v,A')/3\right)\leq 2\exp\left(-k/12\right).
\]
Thus, $\E(|B_i\setminus B'|)\leq 2\exp\left(-k/12\right)|B_i|$. Note that, as $d|B_i|\leq k|A|$, we have $|B_i|\leq 4p|A|$, and, therefore, $\E(|B_i\setminus B'|)\leq 8\exp\left(-k/12\right)p|A|$.

Let $k_0=k/800\log k$, so that $e(G_i)=e(B_i,A)\geq 10k_0|A|$. Note also that every vertex of $B_i\setminus B'$ has degree at most $\Delta(G) \leq k^7$ in $A'$. Hence
$e(G_i[B'\cup A']) \geq e(G_i[B_i,A'])- k^7|B_i\setminus B'|$. Since $\E\big(e(G_i[B_i\cup A'])\big)=p \cdot e(G_i[B_i,A])$ and $\E(|A'|)=p|A|$, we have
\begin{align*}
\E\big(e(G_i[B'\cup A'])-k_0(|A'|+|B'|)\big)&\geq \E\big(e(G_i[B_i,A'])- k^7|B_i\setminus B'|-k_0|A'|-k_0|B_i|\big)\\
&\geq  p\cdot e(G_i[B_i,A])-k^7\cdot 8\exp\left(-k/12\right)p|A|-k_0 \cdot p|A|-k_0\cdot 4p|A|\\
&\geq  p(10k_0|A|-k^{7}\cdot 8\exp\left(-k/12\right)|A|-5k_0|A|)\\
&=  p|A|(5k_0-k^{7}\cdot 8\exp\left(-k/12\right)) \geq 0.
\end{align*}

Thus, there is some choice of $A'$ for which $e(H)-k_0(|A'|+|B'|)\geq 0$. Then $d(H)\geq 2k_0=k/400\log k$ and $\Delta(H)\leq k$ (as explained above), completing the proof.
\end{proof}

Graphs produced by Lemma~\ref{maxdegreduced} have high enough average degree compared to their maximum degree that taking a random subgraph is likely, with a small alteration, to find a dense $C_4$-graph, as follows.

\begin{lemma}\label{randsublocal} Let $G$ be a graph with maximum degree $\Delta$ and average degree $d\geq \Delta^{3/4}$. Then, $G$ contains a $C_4$-free subgraph with average degree at least $d\cdot \Delta^{-3/4}/4$.
\end{lemma}
\begin{proof} Let $n=|G|$.
For every $v\in V(G)$, there are clearly at most $\Delta^3$ paths of length three starting at $v$ and hence at most $\Delta^3$ copies of $C_4$ containing $v$. Thus, $G$ contains at most $n\Delta^3$ copies of $C_4$.
Let $H$ be a random subgraph of $G$ formed by including each edge independently at random with probability $p=\Delta^{-3/4}/2$.
Let $X$ be the number of copies of $C_4$ in $H$. As $e(G)= n\cdot d(G)/2= nd/2$, we have
\[
\E\big(e(H)- X\big)\geq p\cdot e(G)-n \Delta^3 p^4 \geq nd \Delta^{-3/4}/4 -n/16 \geq nd \Delta^{-3/4}/8.
\]
Thus, there is some subgraph $H\subset G$ with $e(H)-X\geq nd \Delta^{-3/4}/8$. Removing an edge from each $C_4$ in $H$ gives a $C_4$-free subgraph with average degree at least $d \Delta^{-3/4}/4$.
\end{proof}

By combining Lemma~\ref{maxdegreduced} and Lemma~\ref{randsublocal}, we can now find a dense $C_4$-free subgraph in graphs satisfying (2) in Lemma~\ref{splitbigdeg}. For convenience, we record this as follows.

\begin{corollary}\label{corbigdiff} The following holds for sufficiently large $k$, and $t\leq k^{1/5}$.
Every bipartite graph $G$ with $d(G)\geq k$ contains either
\begin{enumerate}[label = (\arabic{enumi})]
\item a subgraph $H$ with vertex classes $A$ and $B$ such that $d(H)\geq k/4$, $d_H(v)\leq k$ for each $v\in A$, and $|A|\geq k^6|B|$, or
\item a $C_4$-free subgraph with average degree at least $t$.
\end{enumerate}
\end{corollary}
\begin{proof}
By Lemma~\ref{splitbigdeg}, if (1) does not hold, then $G$ contains a subgraph $H$ which is bipartite with vertex classes $A$ and $B$ such that $d(H)\geq k/4$, $d_H(v)\leq k$ for every $v\in A$ and $\Delta(H)\leq k^7$. By Lemma~\ref{maxdegreduced}, $H$ contains a subgaph $H'$ with maximum degree at most $k$ and average degree at least $k/400\log k$.
Then, letting $\Delta=k$, for sufficiently large $k$, by Lemma~\ref{randsublocal} $H'$ contains a $C_4$-free subgraph with average degree at least $(k/400\log k)\cdot k^{-3/4}/4\geq k^{1/5}\geq t$, so that (2) holds.
\end{proof}

We now find dense $C_4$-free subgraphs in graphs which satisfy (2) in Lemma~\ref{twocases} and are very unbalanced.

\begin{lemma}\label{newdiff} Let $d\geq 2$.
Suppose a bipartite graph $G$  with vertex classes $A$ and $B$ satisfies
\begin{itemize}
\item $|A|\geq d^6|B|$,
\item $d(x,y) \leq d$, for all distinct $x,y\in B$, and
\item  $d_G(v)\leq d$, for every $v\in A$.
\end{itemize}
Then, $G$ contains a subgraph $G'$ which is $C_4$-free and has average degree at least $d(G)/5$.
\end{lemma}
\begin{proof} Let $\ell=d(G)$ and let $A'\subset A$ be a subset of vertices chosen independently at random with probability $p=1/d^6$ and let $G'=G[A'\cup B]$.
Each vertex $v\in B$ has at most $d\cdot d(v)$ paths of length two starting at it. Since $d(v,u) \leq d$ for all $u \in N^2(v)$, $v$ is contained in at most $d^2\cdot d(v)$ copies of $C_4$. Thus, there are at most $\sum_{v\in B}d^2\cdot d(v)\leq d^2\cdot e(G)$ copies of $C_4$ in $G$. Let $X$ be the number of copies of $C_4$ in $G'$, and note that $\E(X)\leq p^2d^2e(G)\leq p e(G)/2$. Also note that $|B|\leq p|A|$, $\E|A'|=p|A|$ and $\ell|A|\leq 2e(G)$. Hence
\[
\E\big(e(G')-X-(|A'|+|B|)\ell/10\big)\geq p e(G)-p e(G)/2-2p|A|\cdot \ell/10= p\big(e(G)/2-\ell |A|/5\big)\geq 0.
\]
Thus, there is some subgraph $G'\subset G$ with $e(G')-X-(|A'|+|B|)\ell/10\geq 0$. Taking $G'$ and removing one edge from each $C_4$ gives a $C_4$-free graph with average degree at least $\ell/5$.
\end{proof}

In our proof, at every iteration (if required), we wish to apply Corollary~\ref{corbigdiff}, then Lemma~\ref{twocases}, then Lemma~\ref{newdiff}. For convenience we combine these steps in the following corollary. 
The improvement we have made can be seen by comparing this to Corollary~\ref{cortwocasesnew}.

\begin{corollary}\label{cortwocases} For sufficiently large $t$, let $G$ be a bipartite graph with $d(G)\geq d\geq t^5$. Then, $G$ contains either
  \begin{enumerate}[label = (\arabic{enumi})]
    \item a $C_4$-free subgraph with average degree at least $t$, or,
    \item a vertex $v$ and sets $A\subset N_G(v)$ and $B\subset V(G)\setminus (A\cup \{v\})$ with $d(G[A,B])\geq d/50t$.
\end{enumerate}
\end{corollary}
\begin{proof} By Corollary~\ref{corbigdiff}, if the required $C_4$-free subgraph does not exist, then $G$ contains a subgraph $H$ with vertex classes $A$ and $B$ such that $d(H)\geq d/4$, $d_H(v)\leq d$ for each $v\in A$, and $|A|\geq d^6|B|$. Apply Lemma~\ref{twocases} to $H$ with $\lambda=d/50t$.
If case (1) in Lemma~\ref{twocases} holds, then (2) holds here. Therefore, assume that there is some spanning subgraph $H'\subset H$ with $d(H')\geq d(H)/(\lambda+1)\geq 6t$ and, $d_{H'}(x,y) \leq \lambda \leq d$
for all $x,y\in B$. By Lemma~\ref{newdiff}, $H'$ contains a $C_4$-free subgraph with average degree at least $t$, as required.
\end{proof}

\begin{proof}[Proof of Theorem~\ref{shortC4free}] Suppose $G$ is bipartite, with average degree at least $(50t)^{9t^2}$ and suppose, for contradiction, that $G$ contains no $C_4$-free subgraph with average degree at least $t$.
For each $i=1,\ldots,4t^2$, by Corollary~\ref{cortwocases}, we can find a vertex $v_i$ and sets $A_i\subset N(v_i)$ and $B_i\subset B_{i-1}\setminus \{v_i\}$ or $B_i\subset A_{i-1}\setminus \{v_i\}$ and a graph $G_i=G_{i-1}[A_i,B_i]$ with average degree at least $(50t)^{9t^2-i}$. Moreover, note that, for each $j<i$, $A_i\subset N(v_j)$ or $B_i\subset N(v_j)$.

Let $s=4t^2$. Thus we have vertices $v_1,\ldots,v_{2s}$ and a graph $G_{2s}$ with average degree at least $(50t)^{t^2}\geq s$ and vertex sets $A_{2s}$ and $B_{2s}$ such that, for each $i\in [s]$, either $A_{2s}\subset N(v_i)$ or $B_{2s}\subset N(v_i)$.  Relabelling, we can assume that we have vertices $v_1,\ldots,v_s$ with $A_s\subset N(v_i)$ for each $i\in [s]$. As $d(G_{2s})\geq s$, $|A_{2s}|\geq s$, and therefore $G_{2s}$, and hence $G$, contains a copy of $K_{s,s}$. Hence, by Proposition~\ref{bipartitecomp}, $G$ contains a $C_4$-free subgraph with average degree at least $t$, a contradiction.
\end{proof}

\section{Lower bound}
In this section we prove Theorem \ref{exampleC4free}. Recall, from the introduction, the following bound of Erd\H{o}s, R\'enyi, and Sos~\cite{ERS} on the Tur\'an number of $C_4$.  It says that every $n$-vertex $C_4$-free graph $G$ has $e(G)\leq n^{3/2}/2+n/4$.
In particular, then, the complete bipartite  graph $K_{n,n}$ has average degree $n$ but no $C_4$-free subgraphs with average degree $\Omega(n^{1/2})$. As noted already in the beginning of the paper, such regular graphs are unlikely to provide good lower bounds. Instead, we base our construction on (highly irregular) graphs without regular subgraphs, constructed by Pyber, R\"odl and Szemer\'edi~\cite{PRS}, as follows.
\begin{theorem}\label{Theorem_PRS_construction}
For all $d$, there exists a graph with $d(G)\geq d$ which contains no $3$-regular subgraph.
\end{theorem}
To proved our lower bound, we take (essentially) the Pyber--R\"odl--Szemer\'edi graph $H$ with parameter $4d$ and blow up every vertex into a set of $d^2$ vertices, replacing each edge by a copy of $K_{d^2, d^2}$. Intuitively, this graph should not have a $C_4$-free subgraph $H'$ with average degree much bigger than $d$. Indeed, firstly, the above estimate for the Tur\'an number of $C_4$ prevents the edges in $H'$ between any two blown-up vertices from having $C_4$-free subgraphs with average degree $\geq d$. 
Moreover, the subgraph of $H$  whose edges corresponds to pairs of blown-up vertices with large average degree between them in $H'$, can be shown to have maximum degree bounded by $d^4$. Then, from  
the properties of the Pyber--R\"odl--Szemer\'edi graph, it follows that such a graph has few edges. To prove this rigorously, we first modify the Pyber--R\"odl--Szemer\'edi graph slightly, using another result from \cite{PRS}.
\begin{theorem}\label{Theorem_PRS_regular_subgraph}
There is a constant $c$ such that the following holds.
Let $G$ be a graph with $\Delta(G)=\Delta$ and $d(G)\geq c\log\Delta$. Then $G$ has a $3$-regular subgraph.
\end{theorem}

\begin{corollary}\label{Corollary_irregular_graph}
There is a constant $c$ so that for all $d\geq 1$ there exists a balanced bipartite graph $H$ with bipartition classes $A, B$  having $d(a)=d$ for all $a\in A$, and 
 which contains no subgraph $\Gamma$ with $d(\Gamma)\geq c\log d$ and $\Delta(\Gamma)\leq d^{4}$.
\end{corollary}
\begin{proof}
Let $c$ be such that Theorem~\ref{Theorem_PRS_regular_subgraph} holds with the constant $c/4$.
Using Theorem~\ref{Theorem_PRS_construction}, let $G$ be a graph with $d(G)\geq 4d$ and no 3-regular subgraph. By Theorem~\ref{Theorem_PRS_regular_subgraph}, $G$ then must have no subgraph $\Gamma$ with $d(\Gamma)\geq (c\log \Delta(\Gamma))/4$. Thus, $G$ has no subgraph $\Gamma$ with $d(\Gamma)\geq c\log d$ and $\Delta(\Gamma)\leq d^{4}$.

Now, to deduce the corollary, we need only show that $G$ has a balanced bipartite subgraph with all vertices in one class having degree $d$. To see this, by iteratively deleting vertices with degree at most $2d$ and using that $d(G)\geq 4d$, first find a subgraph $G_1$ of $G$ with $\delta(G_1)\geq 2d$.
Say $G_1$ has bipartition classes $A$ and $B$, with $|A|\geq |B|$. For each vertex $a\in A$, delete all but $d$ edges adjacent to $a$. Then, delete $|A|-|B|$ vertices from $A$. Note that the final graph, $H$ say, is bipartite, balanced, and every vertex in the class from $A$ has degree $d$.
\end{proof}

As we already explained, our example will be a blow up of the graph from the previous corollary. The following lemma will be used to show that certain $C_4$-free subgraphs of this example have to be sparse. 
The sets $A_i$ and $B_j$ here represent the blown-up vertices.

\begin{lemma}\label{Lemma_partite_graph_degenerate}
Let $k,d\geq 4$.
Let $G$ be a $C_4$-free bipartite graph whose vertex classes are $A$ and $B$. Let $A$ have partition $A_1\cup\ldots \cup A_r$ and $B$ have partition $B_1\cup \ldots \cup B_s$, with $|A_i|, |B_j|\leq d^2$ for all $i\in [r]$ and $j\in [s]$. Define a graph $H$ with the vertex set $\{a_1, \dots, a_r, b_1, \dots, b_s\}$ and $a_ib_j$ an edge exactly if $e(A_i, B_j)>0$. Suppose $H$ is $k$-degenerate. 
Then, $d(G)\leq 13kd$.
\end{lemma}

To show this, we use the following lemma.
\begin{lemma}\label{Lemma_partite_graph_high_degree}
Let $d\geq 3$.
Let $G$ be a $C_4$-free bipartite graph whose vertex classes are $A$ and $B$. Suppose $A$ has a partition $A_1\cup \dots\cup A_m$ with $|A_i|\leq d^2$ for each $i\in [m]$. Suppose that for all $b\in B$ and $i\in [m]$, we
 have $|N(b)\cap A_i|=0$ or $|N(b)\cap A_i|\geq d$. Then $d(G)\leq 18d$.
\end{lemma}
\begin{proof}
 Fix an arbitrary $i\in [m]$. Let $B'$ be the set of vertices $b\in B$ with $|N(b)\cap A_i|\geq d$. As $G$ is $C_4$-free, every pair of vertices in $A_i$ has at most 1 common neighbour in $B$. Thus, as each vertex in $B'$ is adjacent to at least $\binom d2$ different pairs of vertices in $A_i$, we have $\binom {|A_i|}2\geq |B'|\binom d2$. Combining this with $|A_i|\leq d^2$ gives $|B'|\leq \frac{d}{d-1}(|A_i|-1)\leq 2|A_i|$. 
 Since $G$ is $C_4$-free, by the bound of Erd\H{o}s-R\'enyi-Sos~\cite{EH1}, we have 
 \[
 e(A_i, B')\leq \frac{(|A_i|+|B'|)^{3/2}}{2}+\frac {|A_i|+|B'|}4\leq (3|A_i|)^{3/2}.
 \]
 Thus, using $|A_i|\leq d^2$, we get $e(A_i, B)=e(A_i, B')\leq (3|A_i|)^{3/2}\leq  9d|A_i|$.
 
Summing this for all $i\in [m]$, we get $e(G)=\sum_{i=1}^me(A_i,B)\leq 9d|A|$ and hence $d(G)\leq 2e(G)/|A|\leq 18d$.
\end{proof}

Using this lemma, we can now prove Lemma~\ref{Lemma_partite_graph_degenerate}.

\begin{proof}[Proof of Lemma~\ref{Lemma_partite_graph_degenerate}]
For an edge $ab$ with $a\in A_i, b\in B_j$, we say that $ab$ is Type 1 if $|N(a)\cap B_j|\geq d$, Type 2 if $|N(b)\cap A_i|\geq d$, and Type 3 if  neither of these occur. (Note $ab$ may be both Type 1 and Type 2). 
Let $G_1$, $G_2$, and $G_3$ be the subgraphs of $G$ with vertex set $V(G)$  consisting of Type $1, 2$, and $3$ edges respectively.  

Notice that Lemma~\ref{Lemma_partite_graph_high_degree} applies to $G_1$ and $G_2$, so that $d(G_1), d(G_2)\leq 18d$. We claim that $G_3$ is $kd$-degenerate. That is, every subgraph of $G_3$ has a vertex with degree at most $kd$, and thus, by an easy induction, $d(G_3)\leq 2kd$. Therefore, if we can show $G_3$ is $kd$-degenerate, we will have $d(G)\leq d(G_1)+d(G_2)+d(G_3)\leq 36d+2kd\leq 13kd$.

To show $G_3$ is $kd$-degenerate, let $G'$ be an arbitrary subgraph of $G_3$. Let $H'$ be the subgraph of $H$ with $a_ib_j$ an edge exactly if $e_{G'}(A_i, B_j)>0$. Since $H$ is $k$-degenerate, there is a vertex of degree $\leq k$ in $H'$, say $a_i$.
Let $v\in A_i\cap V(G')$. 
Since $d_{H'}(a_i)\leq k$, there are at most $k$ sets $B_j$ with $N_{G'}(v)\cap B_j\neq \emptyset$.
By the definition of $G_3$, $|N_G(v)\cap B_j|\leq d$ for all $j$.  Thus, $d_G'(v)\leq kd$. That is, $G'$ has a vertex with degree at most $kd$, as required.
\end{proof}

We now prove the main result of this section.
\begin{proof}[Proof of Theorem~\ref{exampleC4free}] Note that it is sufficient to find, for all $d\geq 4$, a graph with average degree at least $d^3$ and no $C_4$-free subgraph with average degree at most $15cd\log d$, where $c$ is some fixed constant.

Let $c$ be such that, by Corollary~\ref{Corollary_irregular_graph}, we can find a graph $H$ satisfying the conditions in that corollary. Let the bipartition classes of $H$ be $\{a_1, \dots, a_m\}$ and $\{b_1, \dots, b_m\}$. Let $G$ be the graph formed from $H$ by replacing each  $a_i$ and $b_j$ by a set  $A_i$ and $B_j$ of $d^2$ vertices, respectively, and by replacing each edge $a_ib_j$ of $H$ by a complete bipartite graph $K_{d^2, d^2}$ between $A_i$ and $B_j$. Let $A=A_1\cup \dots \cup A_m$ and $B=B_1\cup \dots \cup B_m$, and note that these sets have size $d^2m$. We have $e(G)=d^4\cdot e(H)= d^5m$, and hence $d(G)= 2e(G)/(2d^2m)=d^3$. 

Now, let $G'$ be a $C_4$-free subgraph of $G$. We will show that $d(G')\leq 15cd\log d$, completing the proof. For this, let $H_1$ be the subgraph of $H$ formed by including the edge $a_ib_j$ exactly if the subgraph
$G'[A_i\cup B_j]$ has a vertex in $A_i$ with at least two neighbours in $B_j$. Let $H_2=H\setminus H_1$. Let $G_1$ and $G_2$ be the subgraphs of $G'$ on vertex set $V(G')$ formed from the union of the pairs $G'[A_i\cup B_j]$ corresponding to edges $a_ib_j$ in $H_1$ and $H_2$ respectively. 

Similarly to the simple argument in the proof of Lemma~\ref{Lemma_partite_graph_high_degree}, since $G'$ is $C_4$-free, for each $j$ there are at most $\binom{|B_j|}2\leq d^4$ vertices $a\in A$ with $\geq 2$ neighbours in $B_j$. Therefore, $d_{H_1}(b_j)\leq d^4$.
From the properties of $H$ from Corollary~\ref{Corollary_irregular_graph}, we have $d_H(a_i)=d$ for all $i\in [m]$, and therefore $\Delta(H_1)\leq d^4$. Again, by the properties of $H$, we have that $d(H_1)<c\log d$. This holds also for any subgraph of $H_1$, since it as has maximum degree at most $d^4$ as well. Hence $H_1$ is $(c\log d)$-degenerate. By Lemma~\ref{Lemma_partite_graph_degenerate}, therefore, $d(G_1)\leq 13cd\log d$.
Notice that, by the definition of $G_2$, each $a\in A_i$ has $d_{G_2}(a)\leq d_H(a_i)=d$. Indeed, by definition, $a$ has in $G_2$ at most one neighbour in every set $B_j$ with $a_ib_j \in E(H)$. Therefore, $d(G_2)\leq 2d$, and hence $d(G')=d(G_1)+d(G_2)\leq 13cd\log d+2d\leq 15cd\log d$, as required.
\end{proof}

\section{Concluding remarks}
\begin{itemize}
\item Denote by $f_6(t)$ the smallest integer such that every graph $G$ with average degree at least $f_6(t)$ contains a subgraph with girth at least 6 and average degree at least $t$. In this paper, we studied the growth rate of $f_6(t)$ and proved that it is at most exponential in $t^{2+o(1)}$.
The best lower bound that we obtained is roughly cubic. It would be very interesting to close this gap further and, in particular, to decide whether  $f_6(t)$ is polynomial in $t$.
\item The next open case of Thomassen's conjecture remains a fascinating, and seemingly very difficult problem. That is, do graphs with large average degree always contain subgraphs with large average degree and girth at least 8?
\item
A related question, posed by Erd\H{o}s and Hajnal \cite{EH1, EH2} in the 1960's, asks whether, for every $k$ and $g$, there is a function $\chi(k,g)$ such that any graph with chromatic number at least $\chi(k,g)$ contains a subgraph with chromatic number at least $k$ and girth greater than $g$. 
Unlike Thomassen's conjecture, here the case $g=3$ is already highly nontrivial, though it was solved by R\"odl~\cite{Ro} using a very elegant argument. All other cases remain open.
\end{itemize}

\end{document}